\documentclass[11pt]{article}

\usepackage{algorithm}
\usepackage{algorithmic}

\usepackage[T1]{fontenc}
\usepackage[utf8]{inputenc}
\usepackage{textcomp}
\usepackage{amssymb}
\usepackage{amsmath,amsfonts}
\usepackage{amsthm}
\usepackage{latexsym}
\usepackage{mathrsfs}
\usepackage{color}
\usepackage{comment}
\usepackage{algorithm}
\usepackage{algorithmic}
\usepackage{subfigure}
\usepackage{dsfont}
\usepackage[title]{appendix}
\usepackage{bbm}
\usepackage{multirow}
\usepackage{varioref}
\usepackage[colorlinks,
            linkcolor=blue,
            anchorcolor=blue,
            citecolor=blue,
            ]{hyperref}
\usepackage{cleveref}
\usepackage{url}            
\usepackage{booktabs}       
\usepackage[a4paper]{geometry}
\usepackage{bm}
\usepackage{graphicx}
\usepackage[numbers]{natbib}
\usepackage{amsopn}

\newcommand{\email}[1]{\protect\href{mailto:#1}{#1}}

\newtheorem{lem}{Lemma}
\newtheorem{theo}{Theorem}
\newtheorem{coro}{Corollary}

\newtheorem{assumption}{Assumption}

\newtheorem{example}{Example}

\newtheorem{remark}{Remark}
\newtheorem{con}{Condition}

\graphicspath{{fig/}}

\newcounter{spb}
\def\1{\mathbbm{1}}

\newcommand{\bw}{\boldsymbol{w}}
\newcommand{\w}{\boldsymbol{w}}
\newcommand{\lam}{\lambda}
\newcommand{\blam}{\boldsymbol{\lambda}}
\newcommand{\bsig}{\boldsymbol{\sigma}}
\newcommand{\bv}{\boldsymbol{v}}
\newcommand{\bell}{\boldsymbol{\ell}}

\newcommand{\ex}{\mathbb{E}}
\newcommand{\+}{_{k+1}}
\newcommand{\m}{_{k-1}}
\newcommand{\0}{_k}
\newcommand{\x}{\boldsymbol{x}}
\newcommand{\y}{\boldsymbol{y}}
\newcommand{\bx}{\boldsymbol{x}}

\newcommand{\la}{\lambda}

\newcommand{\de}{\delta}
\newcommand{\ep}{\epsilon}
\newcommand{\half}{\frac{1}{2}}

\usepackage{amsfonts}

\date{\today}
\author{
Yuze Ge\thanks{School of Data Science, Fudan University, Shanghai, China. (\email{yzge23@m.fudan.edu.cn}).}
\and Rujun Jiang\thanks{Corresponding author. School of Data Science, Fudan University, Shanghai, China. (\email{rjjiang@fudan.edu.cn}).}
}
\title{SOREL: A Stochastic Algorithm for Spectral Risks Minimization}
\usepackage{varioref}

\begin{document}

\maketitle
\begin{abstract}
The spectral risk has wide applications in machine learning, especially in real-world decision-making, where people are not only concerned with models' average performance. By assigning different weights to the losses of different sample points, rather than the same weights as in the empirical risk, it allows the model's performance to lie between the average performance and the worst-case performance. In this paper, we propose SOREL, the first stochastic gradient-based algorithm with convergence guarantees for the spectral risk minimization. Previous algorithms often consider adding a strongly concave function to smooth the spectral risk, thus lacking convergence guarantees for the original spectral risk. We theoretically prove that our algorithm achieves a near-optimal rate of $\widetilde{O}(1/\sqrt{\ep})$ in terms of $\ep$. Experiments on real datasets show that our algorithm outperforms existing algorithms in most cases, both in terms of runtime and sample complexity.
\end{abstract}

\section{Introduction}
In modern machine learning, model training heavily relies on minimizing the empirical risk. This ensures that the model have high average performance.
Given a training set of $n$ sample points $\mathcal{D}=\{{\boldsymbol{x}_i}\}_{i=1}^n\subset\mathcal{X}$, the goal of the empirical risk minimization is to solve
$$R(\bw) = (1/n)\sum_{i=1}^n \ell_i(\bw).$$
Here, $\bw\in\mathbb{R}^d$ is the parameter vector of the model, $\ell_i(\bw)=\ell(\bw,\boldsymbol{x}_i) $ is the loss of the $i$-th sample, and $\ell:\mathbb{R}^d\times \mathcal{X}\to \mathbb{R}$ is the loss function. However, as machine learning models are deployed in real-world scenarios, the evaluation metrics for these models become more diverse, including factors such as fairness or risk aversion.

In this paper, we consider a generalized aggregation loss function: the spectral risk, which is of the form
$$R_{\bsig} (\bw) = \sum_{i=1}^n \sigma_i \ell_{[i]}(\bw).$$
Here $\ell_{[1]}(\cdot) \leq \dots \leq \ell_{[n]}(\cdot)$ denotes the order statistics of the empirical loss distribution, and $0\leq\sigma_1 \leq \dots \leq \sigma_n, \sum_{i=1}^n \sigma_i = 1.$

In form, the spectral risk penalizes the occurrence of extreme losses by assigning greater weights to extreme losses. When $\sigma_i = 1/n$, the spectral risk reduces to the empirical risk. When $\sigma_n = 1$ and $\sigma_i = 0$ for $i = 1, \dots, n-1$, the spectral risk becomes the maximum loss. Therefore, the spectral risk measures the model's performance between the average case and the worst case. By assigning different values to $\sigma_i$, the spectral risk encompasses a wide range of aggregated loss functions that have broad applications in fields such as machine learning and finance.
Common spectral risks include Conditional Value at Risk (CVaR) or the average of top-k loss~\citep{artzner1997thinking,rockafellar2000optimization}, Exponential Spectral Risk Measure (ESRM)~\citep{cotter2006extreme}, and Extremal Spectral Risk (Extremile)~\citep{daouia2019extremiles}. Their specific forms are shown in Table~\ref{tab:sigma}~\citep{mehta2022stochastic}.

\begin{table}[t]
\caption{
Different spectral risks and the corresponding weights.}
\label{tab:sigma}
\begin{center}
\begin{tabular}{lcl}
\toprule
\multicolumn{1}{c}{\bf Spectral Risks} &\multicolumn{1}{c}{\bf Parameter}  &\multicolumn{1}{c}{ $\sigma_i$}
\\ \hline \\
$\alpha$-CVaR   & $0<\alpha<1$     & \hspace{0.5em} $\begin{cases}\frac{1}{n \alpha}, & i> \lceil n(1-\alpha)\rceil \\ 1-\frac{\lfloor n \alpha\rfloor}{n \alpha}, &\lfloor n(1-\alpha)\rfloor<i<\lceil n(1-\alpha)\rceil\\ 0, &\text{otherwise}\end{cases}$\\
$\rho$-ESRM     &$\rho>0$        &\quad $e^{-\rho}\left(e^{\rho\frac{i}{n}}-e^{\rho\frac{i-1}{n}}\right)/\left(1-e^{-\rho}\right)$\\
$r$-Extremile   & $r\geq 1$         &\hspace{0.7em}$\left(\frac{i}{n}\right)^r-\left(\frac{i-1}{n}\right)^r$ \\
\bottomrule
\end{tabular}
\end{center}
\end{table}

Despite the broad applications of spectral risks, optimization methods for spectral risks are still limited. In particular, for large-scale optimization problems, there is currently a lack of stochastic algorithms with convergence guarantees for the spectral risk minimization.
Indeed, the weight of each sample point depend on the entire training set, introducing complex dependencies and thus making the optimization process challenging.
Existing algorithms either abandon the convergence guarantee to the minimum of the spectral risk problem due to the difficulty of obtaining unbiased subgradient estimates~\citep{levy2020largescale,kawaguchi2020ordered}, or turn to optimize the smooth regularized spectral risk~\citep{mehta2024distributionally,mehta2022stochastic}, which lacks convergence guarantees for the original spectral risk. Given the widespread application of the spectral risk in machine learning and the lack of stochastic algorithms for the spectral risk minimization, we are committed to developing stochastic algorithms with convergence guarantees for the spectral risk minimization.

\paragraph{Our Contributions.}

  In this paper, we study \textbf{S}tochastic \textbf{O}ptimization for Spectral \textbf{R}isks with traj\textbf{e}ctory Stabi\textbf{l}ization (SOREL). i) We propose SOREL, the first stochastic algorithm with convergence guarantees for the spectral risk minimization. In particular, SOREL stabilizes the trajectory of the primal variable to the optimal point. ii) Theoretically, we prove that SOREL achieves a near-optimal rate of $\widetilde{O}(1/\sqrt{\ep})$ in terms of $\ep$ for spectral risks with a strongly convex regularizer. This matches the known lower bound of $\Omega(1/\sqrt{\ep})$ in the deterministic setting~\citep{ouyang2021lower}. iii) Experimentally, SOREL outperforms existing baselines in most cases, both in terms of runtime and sample complexity.

\section{Related work}

\paragraph{Statistical Properties of the Spectral Risk.} \  As a type of risk measures, the spectral risk assigns higher weights to the tail distribution and has been profoundly studied in the financial field~\citep{artzner1999coherent,rockafellar2013fundamental,he2022risk}. Recently, statistical properties of the spectral risk have been investigated by many works in the field of learning theory. Specifically, \cite{mehta2022stochastic,undefined2022Wasserstein} demonstrated that the discrete form of spectral risks converges to the spectral risk of the overall distribution, controlled by the Wasserstein distance. \cite{holland2022spectral,khim2020uniform,holland2021learning} also considered the learning bounds of spectral risks. 

\paragraph{Applications.} \  The spectral risk is widely applied in various fields of finance and machine learning. In some real-world tasks, the worst-case loss is as important as the average-case loss, such as medical imaging~\citep{xu2022clinicalrealistic} or robotics \citep{sharma2020riskaware}. The spectral risk minimization can be viewed as a risk-averse learning strategy. In the domain of fair machine learning, different subgroups are classified by sensitive features (e.g., gender and race). Subgroups with higher losses may be treated unfairly in decision-making. By penalizing samples with higher losses, the model's performance across different subgroups can meet certain fairness criteria~\citep{williamson2019fairness}, such as demographic parity \citep{dwork2012fairness} and equalized odds \citep{hardt2016equality}.
In the field of distributionally robust optimization, the sample distribution may deviate from a uniform distribution, which can be modeled by reweighting the samples \citep{chen2020distributionally}. \cite{mehta2024distributionally} adopts the polytope as the uncertainty set of the shifted distribution, which is similar to the form of the spectral risk minimization.

In practical applications, we can choose different types of spectral risks based on actual needs. For example, CVaR is popular in the context of portfolio optimization~\citep{rockafellar2000optimization}, as well as reinforcement learning~\citep{zhang2024cvarconstrained,chow2017riskconstrained}. Other applications of spectral risks include reducing test errors and  mitigating the impact of outliers~\citep{maurer2021robust, kawaguchi2020ordered,fan2017learning}, to name a few.

\paragraph{Existing Optimization Methods.} \ There have been many algorithms to optimize CVaR, a special case of spectral risks, including both deterministic \citep{rockafellar2000optimization} and stochastic algorithms \citep{fan2017learning,curi2020adaptive}. For the spectral risk, deterministic methods such as subgradient methods have convergence guarantees, although they are considered algorithms with slow convergence rate. \citet{xiao2023unified} proposed an Alternating Direction Method of Multipliers (ADMM) type method for the minimization of the rank-based loss.
 Other deterministic methods reformulate this problem into a minimax problem~\citet{thekumparampil2019efficient,hamedani2021primaldual,khalafi2023accelerated}. However, these methods require calculating
$O(n)$ function values and gradients at each iteration, posing significant limitations when solving large-scale problems.

Stochastic algorithms for solving the spectral risk minimization problems are still limited. Some algorithms forgo convergence to the true minimum of the spectral risk \citep{levy2020largescale,kawaguchi2020ordered}. Other methods modify the objective to minimize a smooth approximation of the spectral risk by adding a strongly concave term with a coefficient $\nu$ \citep{mehta2022stochastic,mehta2024distributionally}. The smaller $\nu$ is, the closer the approximation is to the original spectral risk. \cite{mehta2024distributionally} proposed the Prospect algorithm and proved that it achieves linear convergence for any
$\nu>0$. Furthermore, if the loss of each sample is different at the optimal point, then the optimal value of the smooth approximation of the spectral risk is the same as the optimal value of the original spectral risk
as long as
$\nu$ is below a certain positive threshold. However, in practice, these conditions are difficult to verify. The convergence of these algorithms still lacks guarantees for the original spectral risk minimization.
Other stochastic algorithms include \citet{hamedani2023stochastic,yan2019stochastic}, which have a slower convergence rate of $O(1/\ep)$ in terms of $\ep$. In this paper, we propose SOREL for the original spectral risk minimization problems and achieve a near-optimal convergence rate in terms of
 $\ep$.

\section{Algorithm}
We consider stochastic optimization of the spectral risk combined with a strongly convex regularizer:
\begin{equation}\label{eq:problem}
    \begin{aligned}
        \min_{\bw} \underbrace{\sum_{i=1}^n \sigma_i \ell_{[i]}(\bw)}_{R_{\bsig}(\bw)}+g(\bw).
    \end{aligned}
\end{equation}

Firstly, we make the basic assumption about the individual loss function $\ell_i$ and the regularizer $g$.
\begin{assumption}\label{assumption:basic}
    The individual loss function $\ell_i:\mathbb{R}^d\to\mathbb{R}$ is convex, $G$-Lipschitz continuous and $L$-smooth for all $i\in\{1,\dots,n\}$. The regularizer $g:\mathbb{R}^d\to\mathbb{R}\cup\{\infty\}$ is proper, lower semicontinuous and $\mu$-strongly convex.
\end{assumption}
Assumption~\ref{assumption:basic} is common in the literature on stochastic optimization~\citep{nemirovski2009robust,davis2019stochastic}, especially in the field of the spectral risk minimization~\citep{kawaguchi2020ordered,holland2022spectral,levy2020largescale,mehta2022stochastic, mehta2024distributionally}. Both the logistic loss and the least-square loss within bounded sublevel sets satisfy this assumption. The $L_2$ regularization, which is one of the most common regularizers in machine learning, satisfies the assumption on $g$.

\subsection{Challenges of Stochastic Optimization for Spectral Risks}\label{sec:alg-prob}

In this section, we describe the challenges of the spectral risk minimization problem and the techniques to solve them.

\paragraph{Biases of Stochastic Subgradient Estimators.}
From convex analysis~\citep[Lemma 10]{wang2023proximal}, we know that the subgradient of $R_{\bsig}$ is
$$\partial R_{\bsig}(\bw)=\operatorname{Conv}\left\{\bigcup_{\pi}\left\{\sum_{i=1}^n\sigma_i\nabla \ell_{\pi(i)}(\bw):\ell_{\pi(1)}(\bw)\leq\dots\leq \ell_{\pi(n)}(\bw)\right\}\right\},$$
where $\operatorname{Conv}$ denotes the convex hull of a set, and $\pi$ is a permutation that arranges $\ell_1,\dots,\ell_n$ in ascending order. Note that $R_{\bsig}(\bw)$ is non-smooth. Indeed, when there exist $i\ne j$ such that $\ell_i(\bw)=\ell_j(\bw)$, $\partial R_{\bsig}(\bw)$ contains multiple elements.

The subgradient of $R_{\bsig}$ is related to the ordering of $\ell_1, ..., \ell_n$. We cannot obtain an unbiased subgradient estimator of $\partial R_{\bsig}$ if we use only a mini-batch with $m$ $(m<n)$ sample points. For example, when $m = 1$, we randomly sample $i$ uniformly from $\{1,\dots,n\}$. The subgradient estimator $\nabla \ell_i(\bw)$ is unbiased only if $\sigma_i = 1/n$.
For general $\sigma$, unfortunately, to obtain an unbiased subgradient estimator of $\partial R_{\bsig}$,
we have to compute $n$ loss function values and then determine the ranking of $\ell_i$ among the $n$ losses (or the weight corresponding to the $i$-th sample point). However, computing $O(n)$ losses at each step is computationally heavy. To remedy this, we next design an algorithm that first use a minimax reformulation of problem \eqref{eq:problem} and then alternates updates $\blam$ and $\bw$ using a primal-dual method.

Equivalently, we can rewrite $R_{\bsig}(\bw)$ in the following form
\begin{equation}\label{eq:spectral_risk_reform}
    \begin{aligned}
        R_{\bsig}(\bw) = \max_{\blam\in \Pi_{\bsig}} \sum_{i=1}^n \lambda_i \ell_{i}(\bw),
    \end{aligned}
\end{equation}
where $\Pi_{\bsig}= \{\Pi\bsig: \Pi \mathbf{1}=\mathbf{1}, \Pi^{\top} \mathbf{1}=\mathbf{1}, \Pi \in[0,1]^{n \times n}\}$ is the permutahedron  associated with $\bsig$, i.e., the convex hull of all permutations of $\bsig$, and $\mathbf{1}$ is the all-one vector~\citep{blondel2020fast}.  Then Problem (\ref{eq:problem}) can be rewritten as
\begin{equation}\label{eq:problem_reform}
    \begin{aligned}
        \min_{\bw}\max_{\blam\in \Pi_{\bsig}} L(\w,\blam)=\sum_{i=1}^n \lambda_i \ell_{i}(\bw)+g(\bw).
    \end{aligned}
\end{equation}

Next, we use a primal-dual method to solve Problem (\ref{eq:problem_reform}). Specifically, we iteratively update $\bw$ and $\blam$:
\begin{equation}\label{eq:lam_sub_prob}
    \blam_{k+1}=\underset{\blam\in\Pi_{\bsig}}{\arg\max}\sum_{i=1}^n \lam_i\ell_i(\bw_k)-\frac{1}{2\eta_k}\|\blam-\blam_k\|^2,
\end{equation}
\begin{equation}\label{eq:w_sub_prob}
    \bw_{k+1}=\underset{\bw}{\arg\min}P_k(\bw):=\sum_{i=1}^n\lam_{i,k+1}\ell_i(\bw)+g(\bw)+\frac{1}{2\tau_k}\|\bw-\bw_k\|^2.
\end{equation}

Steps (\ref{eq:lam_sub_prob}) and (\ref{eq:w_sub_prob}) can be seen as alternatingly solving the min problem and the max problem in (\ref{eq:problem_reform}) with proximal terms.

\paragraph{Stabilizing the Optimization Trajectory.}
To update $\blam_{k+1}$, one may naturally think of solving Problem~(\ref{eq:spectral_risk_reform}):  $\blam_{k+1}=\arg\max_{\blam\in\Pi_{\bsig}}\sum_{i=1}^n\lam_i\ell_i(\bw_k)$, similar to methods in \citet{mehta2022stochastic,mehta2024distributionally} with smoothing coefficient $\nu=0$. However, since Problem (\ref{eq:spectral_risk_reform}) is merely convex, the solution
$\blam$ lacks continuity with respect to $\bw$, that is, a small change in $\bw$ could lead to a large change in $\blam$. Indeed, it is often the case that there are multiple optimal solutions for (\ref{eq:spectral_risk_reform}) when there exists $i\ne j$ such that $\ell_i(\bw)=\ell_j(\bw)$, and in this case, an arbitrary small perturbation of $\bw$ will lead to a different value of $\lambda_i$.
{\ As shown in Figure~\ref{fig:pre-stable}, this can cause
$\bw$ to oscillate near points where some losses are the same and prevents the convergence of the algorithm. We also provide a toy example in Appendix~\ref{apx:example} to further illustrate this difficulty.} Therefore, the proximal term $\frac{1}{2\eta_k} \|\blam-\blam_k\|^2$ is added in (\ref{eq:lam_sub_prob}) to prevent excessive changes in $\blam$ and stabilize the trajectory of the primal variable, where $\eta_k>0$ controls the extent of its variation.

\paragraph{Stochastic Optimization for the Primal Variable.}
We use a stochastic algorithm to approximately solve (\ref{eq:w_sub_prob}). Through the minimax reformulation in (\ref{eq:w_sub_prob}), we avoid directly calculating the stochastic subgradient of $R_{\bsig}(\bw)$, which requires computing all loss function values to obtain the corresponding sample weight $\lam_i$. Additionally, since $\blam_{k+1}$ is fixed, the finite sum part of the objective function in (\ref{eq:w_sub_prob}) is smooth, allowing us to use variance reduction (VR), a commonly used technique in stochastic optimization~\citep{shalev2013stochastic,roux2012stochastic,johnson2013accelerating,defazio2014saga}, to accelerate our stochastic algorithm. In contrast, since $R_{\bsig}(\bw)$ is non-smooth, as previously mentioned, VR cannot be used to directly solve Problem (\ref{eq:problem}). For smooth convex functions in the form of the finite sum, many methods such as SVRG~\citep{johnson2013accelerating}, SAGA~\citep{defazio2014saga}, and SARAH~\citep{nguyen2017sarah} can enable stochastic methods to achieve the convergence rate of deterministic methods. We apply the proximal
stochastic gradient descent with a generalized VR method inspired by SVRG to approximately solve (\ref{eq:w_sub_prob}), which will be presented in Section~\ref{sec:alg-algorithm} in detail. Thanks to its strong convexity, Problem (\ref{eq:w_sub_prob}) can be solved efficiently.

Similar to (\ref{eq:lam_sub_prob}), we add a proximal term $\frac{1}{2\tau_k}\|\bw-\bw^k\|^2$ in (\ref{eq:w_sub_prob}) where $\tau_k>0$ is the proximal parameter. The proximal parameter $\tau_k$ is crucial for the convergence proof of our algorithm.  By carefully choosing $\tau_k=O(1/k)$, the updates of $\bw$ become more stringent as the algorithm progresses, and SOREL can achieve a near optimal rate of $\widetilde{O}(1/\sqrt{\ep})$ in terms of $\ep$.

\begin{figure}[t]
    \centering
    \includegraphics[width=0.6\linewidth]{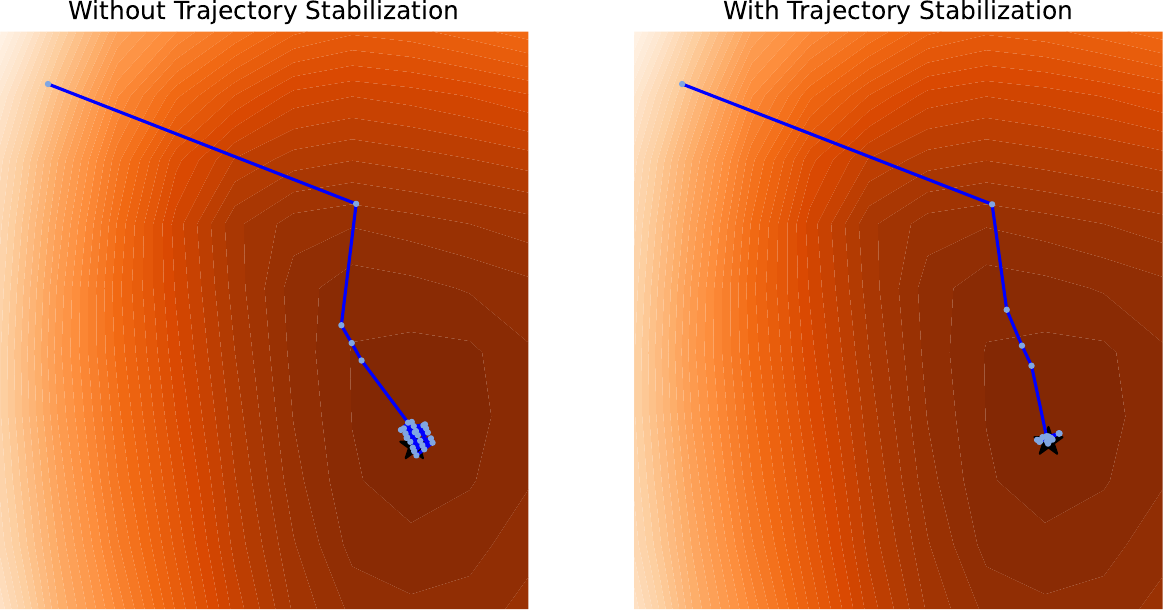}
    \caption{The level set plot of 2D least-square regression with primal-dual optimization trajectories described in Section~\ref{sec:alg-prob}. The max subproblem does not have a proximal term (\textbf{left}) or has a proximal term (\textbf{right}). The min subproblem does not have a proximal term. The black star represents the optimal point. The sample points are obtained by projecting the \texttt{yacht} dataset onto $\mathbb{R}^2$ using PCA.}
    \label{fig:pre-stable}
\end{figure}

\subsection{The SOREL Algorithm}\label{sec:alg-algorithm}
Our proposed algorithm SOREL is summarized in Algorithm~\ref{alg:main_alg}. The specific values for the parameters $\theta_k,\eta_k,\tau_k$ and $m_k$ in  Algorithm~\ref{alg:main_alg} will be given in Section~\ref{sec:theory}. In Line~\ref{algline: inital_lam} the algorithm initializes $\blam_0$ by solving Problem (\ref{eq:spectral_risk_reform}). In Lines \ref{algline:inner_loop_start}-\ref{algline:w_update}, the algorithm computes the stochastic gradient and update $\bw$ for a fixed $\blam$, as described in Section~\ref{sec:alg-prob}. Additionally, we compute the full gradient of $\bw$ every $m_k$ updates to reduce the variance. In Lines \ref{algline:moment}-\ref{algline:update_lam}, we update $\blam$. Note that we replace $\ell_i(\bw_k)$ with $\ell_i(\bw_k)+\theta_k\left(\ell_i(\bw_k)-\ell_i(\bw_{k-1})\right)$ to accelerate the algorithm. This can be seen as a momentum term, a widely used technique in smooth optimization~\citep{tseng1998incremental,liu2020improved,gitman2019understanding,sutskever2013importance}, where $\theta_k>0$ is the momentum parameter.

\begin{algorithm}[H]
\caption{SOREL}
\label{alg:main_alg}
\begin{algorithmic}[1]
\REQUIRE initial $\bw_0$, $\bw_{-1}=\bw_0$, $\bsig$, and learning rate $\alpha$, $\{\theta_k\}_{k=0}^{K-1}$, $\{\eta_k\}_{k=0}^{K-1}$, $\{\tau_k\}_{k=0}^{K-1}$, $\{m_k\}_{k=0}^{K-1}$ and $\{T_k\}_{k=0}^{K-1}$.
\STATE $\blam_0=\arg\min_{\blam\in\Pi_{\bsig}}-\bell(\bw_0)^\top\blam$. \label{algline: inital_lam}
\FORALL{$k=0,\dots,K-1$}\label{algline:inner_loop}
    \STATE $\bv_k=(1+\theta_k)\bell(\bw_k)-\theta_k\bell(\bw_{k-1})$. \label{algline:moment}
    \STATE $\blam_{k+1}=\arg\min_{\blam\in\Pi_{\bsig}} -\bv_k^\top\blam+\frac{1}{2\eta_k}\|\blam-\blam^k\|^2$.\label{algline:update_lam}
    \STATE $\bw_{k,0}=\bw_k$, $\Bar{\bw}=\bw_k$.
    \STATE $\bar{\boldsymbol{g}}=\sum_{i=1}^n\lambda_{i,k+1}\nabla\ell_i(\bar{\bw})$.
    \FORALL{$t=1,\dots,T_k$}\label{algline:inner_loop_start}
        \IF{$t \operatorname{mod} m_k =0$}\label{algline:inner_inner_loop_start}
            \STATE $\Bar{\bw}=\frac{1}{m_k}\sum_{j=t-m_k+1}^t\bw_{k,j}$.\label{algline:primal_epoch_average}
            \STATE $\bar{\boldsymbol{g}}=\sum_{i=1}^n\lambda_{i,k+1}\nabla
            \ell_i(\bar{\bw})$.\label{algline:primal_epoch_fullgrad}
        \ENDIF
        \STATE Sample $i_t$ uniformly from $\{1,\dots,n\}$,
        \STATE $\boldsymbol{d}_{k,t}=n\lambda_{i_t,k+1}\nabla \ell_{i_t}(\bw_{k,t})-n\lambda_{i_t,k+1}\nabla \ell_{i_t}(\Bar{\bw})+\bar{\boldsymbol{g}}$
        \STATE $\bw_{k,t+1}=\operatorname{Prox}_{\alpha \left(g+\frac{1}{2\tau_k}\|\cdot-\bw_k\|^2\right)} \left\{\bw_{k,t}-\alpha\boldsymbol{d}_{k,t}\right\}.$\label{algline:w_update}
        \ENDFOR\label{algline:inner_loop_end}
        \STATE $\bw_{k+1}=\frac{1}{m_k}\sum_{j=T_k-m_k+1}^{T_k}\bw_{k,j}$.\label{algline:inner_loop_output}
\ENDFOR
\ENSURE $\bw_K$.
\end{algorithmic}
\end{algorithm}

Define the proximal operator $\operatorname{prox}_h(\bar{\boldsymbol{x}}):=\arg\min_{\boldsymbol{x}} h(\boldsymbol{x})+\frac{1}{2}\|\boldsymbol{x}-\bar{\boldsymbol{x}}\|^2$ for a function $h$. In Line~\ref{algline:w_update}, we apply the proximal 
stochastic gradient descent step. We assume that $\operatorname{prox}_{g+\frac{1}{2}\|\cdot\|^2}(\cdot)$ is easy to compute, which is the case for many commonly used regularizers $g$, such as the 
$l_1$ norm. If $g$ is differentiable, we can replace the proximal stochastic gradient with stochastic gradient:  $\bw_{k,t+1}=\bw_{k,t}-\alpha\left(\boldsymbol{d}_{k,t}+\frac{1}{\tau_k}\left(\bw_{k,t}-\bw_k\right)+\nabla g(\bw_{k,t})\right).$
{This will not affect the convergence or convergence rate of the algorithm as long as $\nabla g$ is Lipschitz continuous and the step size $\alpha$ is small enough}.
In Line \ref{algline:update_lam}, we need to compute the projection onto $\Pi_{\bsig}$. For an ordered vector, projecting onto the permutahedron takes $O(n)$ operations using the Pool Adjacent Violators Algorithm (PAVA)\citep{lim2016efficient}. In SOREL, we need to first sort $n$ elements of the projected vector and then compute the projection onto $\Pi_{\bsig}$, which takes a total of $O(n \log n)$ operations.

 In practice, we set $T_k$ and $m_k$ to $n$ in Lines~\ref{algline:inner_loop_start} and \ref{algline:inner_inner_loop_start}, meaning the algorithm updates $\blam$ once it traverses the training set. 
We also set the reference point $\bar{\bw}$ and the output $\bw_{k+1}$ in Lines~\ref{algline:primal_epoch_average} and \ref{algline:inner_loop_output} to be the last vector of the previous epoch rather than the average vector, as with most practical algorithms~\citep{johnson2013accelerating,zhu2016variance,cutkosky2019momentumbased,babanezhad2015stopwasting,gower2020variancereduced}.
In this way, SOREL only requires computing the full batch gradient once for each update of $\blam$, and becomes single-loop in Lines~\ref{algline:inner_loop_start}-~\ref{algline:inner_loop_end}. This makes the algorithm more concise and parameters easier to tune.

\section{Theoretical Analysis}\label{sec:theory}

For convenience, we consider that $T_k$ (will be determined in Theorem~\ref{thm:main-thm-paper}) is large enough so that $\w_k$ is a $\delta_k$-optimal solution of $P_k(\w)$, that is, $\ex_k P_k(\w\+)-\min_{\w}P_k(\w)\leq\delta_k$. Here, $\ex_k$ represents the conditional expectation with respect to the random sample points used to compute $\w\+$ given $\w\0,\dots,\w_0$. Then, we can provide a one-step analysis of the outer loop of SOREL. We use $L(\bw,\blam)=\blam^\top\bell(\bw)+g(\bw)$ in the analysis for simplicity.

\begin{lem}\label{lem:one-step-paper}
    Suppose Assumption~\ref{assumption:basic} holds. Let $\{\bw\0\}$ and $\{\blam\0\}$ be the sequences generated by Algorithm~\ref{alg:main_alg}. Then for any $\bw\in\mathbb{R}^d$ and $\blam\in \Pi_{\bsig}$, the following inequality holds,
\begin{equation}\label{eq:one-step}
    \begin{aligned}
        &\ex_k \, \left\{L(\bw\+,\blam)-L(\bw,\blam\+)\right\}\\
        \leq &\ex_k\left\{\langle\blam-\blam\+,\bell(\bw\+)\rangle+\frac{1}{2\eta_k}\left[\|\blam-\blam\0\|^2-\|\blam-\blam\+\|^2-\|\blam\+-\blam\0\|^2\right]\right.\\
        &+\frac{1}{2\tau\0}\left[\|\bw-\bw\0\|^2-\|\bw-\bw\+\|^2-\|\bw\+-\bw\0\|^2\right]-\frac{\mu}{2}\|\bw-\bw\+\|^2\\
        &+ \left.\langle \bv\0,\blam\+ -\blam \rangle +\delta\0+\sqrt{\frac{(\tau_k^{-1}+\mu)\delta_k}{2}}(\|\bw-\bw\+\|^2+1)\right\}.
    \end{aligned}
\end{equation}
\end{lem}

Next, we try to telescope the terms on the right hand side of (\ref{eq:one-step}) by multiplying each term by $\gamma_k$.
By choosing appropriate parameters in Algorithm~\ref{alg:main_alg} to satisfy some conditions (will be discussed in Appendix~\ref{apx:proof}), we can ensure that the adjacent terms indexed by $k=0,\dots,K-1$ can be canceled out during summation. Then we can achieve the desired convergence result.

\begin{theo}\label{thm:main-thm-paper}
    Suppose Assumption~\ref{assumption:basic} holds.
    Set $\gamma_k=k+1,~\eta_k=\frac{\mu(k+1)}{8G^2},~\theta_k=\frac{k}{k+1},~\tau_k=\frac{4}{\mu (k+1)},~    \delta_k = \min \left(\frac{\mu}{8(k+5)},
        \mu(k+1)^{-6}\right)$,
 the step-size $\alpha=\frac{1}{12L}$, $m_k=\frac{384L}{(k+5)\mu}+2$ and $T_k=O(m_k\log\frac{1}{\delta_k})$ in Algorithm~\ref{alg:main_alg}. Then we have
     $$
     \ex \|\bw_K-\bw^\star\|^2=O\left(\frac{G^2}{\mu^2K^2}\right).$$
\end{theo}

\begin{coro}
     Under the same conditions in Theorem~\ref{thm:main-thm-paper}, we obtain an output $\bw_K$ of Algorithm~\ref{alg:main_alg} such that $\ex\|\bw_K-\bw^\star\|^2\leq\epsilon$ in a total sample complexity of  $O\left(n\frac{G}{\mu\sqrt{\epsilon}}\log\frac{G}{\mu^2\sqrt{\epsilon}}+\frac{L}{\mu}\log\frac{G}{\mu\sqrt{\epsilon}}\log\frac{G}{\mu^2\sqrt{\epsilon}}\right)$.
\end{coro}
Our algorithm achieves a near-optimal convergence rate of $\widetilde{O}(1/\sqrt{\ep})$ in terms of $\ep$, which matches the lower bound of $\Omega(1/\sqrt{\ep})$ in the deterministic setting up to a logarithmic term~\citep{ouyang2021lower}. This is the first near-optimal method for solving the spectral risk minimization. Previously, \citet{mehta2022stochastic,mehta2024distributionally} add a strongly concave term with respect to $\blam$ in $L(\bw,\blam)$ and achieve a linear convergence rate for the perturbed problem.
One may set the coefficient of the strongly concave term $\nu$ to $O(\ep)$, obtaining an $\ep$-optimal solution for the original spectral risk minimization problem. However, this approach has drawbacks: it leads to a worse sample complexity of $\widetilde{O}(1/\ep)$~\citep{palaniappan2016stochastic} or even $\widetilde{O}(1/\ep^3)$~\citep{mehta2024distributionally}; additionally, to achieve an $\ep$-optimal solution, the step size would need to be set to $O(\ep)$, resulting in very small steps that perform poorly in practice. In contrast,  SOREL's step size is independent of $\ep$.

\begin{remark}
\normalfont 
    In Lines~\ref{algline:primal_epoch_average} and \ref{algline:inner_loop_output} of Algorithm~\ref{alg:main_alg}, we set the reference point $\bar{\bw}$ and the output $\bw_{k+1}$ to the average of the previous epoch. Instead, we can {also} set them to be the last vector of the previous epoch, which aligns with practical implementation. For theoretical completeness, we may compute the full gradient $\bar{\boldsymbol{g}}$ in Line~\ref{algline:primal_epoch_fullgrad} at each step $t$ with probability $p$ instead of once per epoch (every $m_k$ steps), as done in~\citep{kulunchakov2019estimate, hofmann2015variance, kovalev2020dont}. However, these methods are beyond the scope of this paper.
\end{remark}

\section{Experiments}

In this section, we compare our proposed algorithm SOREL with existing baselines for solving the spectral risk minimization problem.

We consider the least squares regression problem using a linear model with $L_2$ regularization. We adopt a wide variety of spectral risks, including ESRM, Extremile, and CVaR. Five tabular regression benchmarks are used for the least squares loss: \texttt{yacht}~\citep{tsanas2012accurate}, \texttt{energy}~\citep{baressi2019artificial}, \texttt{concrete}~\citep{yeh2006analysis}, \texttt{kin8nm}~\citep{akujuobi2017delve}, \texttt{power}~\citep{tüfekci2014prediction}.
We compare the suboptimality versus passes (the number of samples divided by $n$) and runtime.
The suboptimality is defined as
$$
\text{Suboptimality}(\bw_k)=\frac{R_{\bsig}(\bw_k)+g(\bw_k)-R_{\bsig}(\bw^\star)-g(\bw^\star)}{R_{\bsig}(\bw_0)+g(\bw_0)-R_{\bsig}(\bw^\star)-g(\bw^\star)},$$
where $\bw^\star$ is calculated by L-BFGS~\citep{nocedal1999numerical}.

Baseline methods include SGD~\citep{levy2020largescale,mehta2022stochastic} with a minibatch size of 64, LSVRG~\citep{mehta2022stochastic}, and Prospect~\citep{mehta2024distributionally}. Note that although both LSVRG and Prospect add a strongly concave term with coefficient $\nu$ to smooth the original spectral risk, they have been observed to exhibit linear convergence for the original spectral risk minimization problem in practice without the strongly concave term~\citep{mehta2022stochastic,mehta2024distributionally}. Consequently, we set $\nu=0$ in our experiments. By~\citep{mehta2024distributionally}, if the losses at the optimal point are different from each other, then as long as $\nu$ is below a certain positive threshold, the optimal solutions of the smoothed and original spectral risks minimization problems are the same. For LSVRG, we set the length of an epoch to $n$. For SOREL, we set $T_k=m_k=n$, and in all experiments, the values of $\tau_k$ are fixed and $\theta_k=k/(k+1)$. Thus, our algorithm has only two hyperparameters $\eta_k$ and the learning rate $\alpha$ to tune. {We apply stochastic gradient descent to solve (\ref{eq:w_sub_prob}) instead of proximal stochastic gradient descent.}
We also empirically compare different algorithms with minibatching, where the batch size for all algorithms are set to 64.
Detailed experimental settings are provided in Appendix~\ref{apx:exp-detail}.

\begin{figure}[ht]
    \centering
    \includegraphics[width=\linewidth]{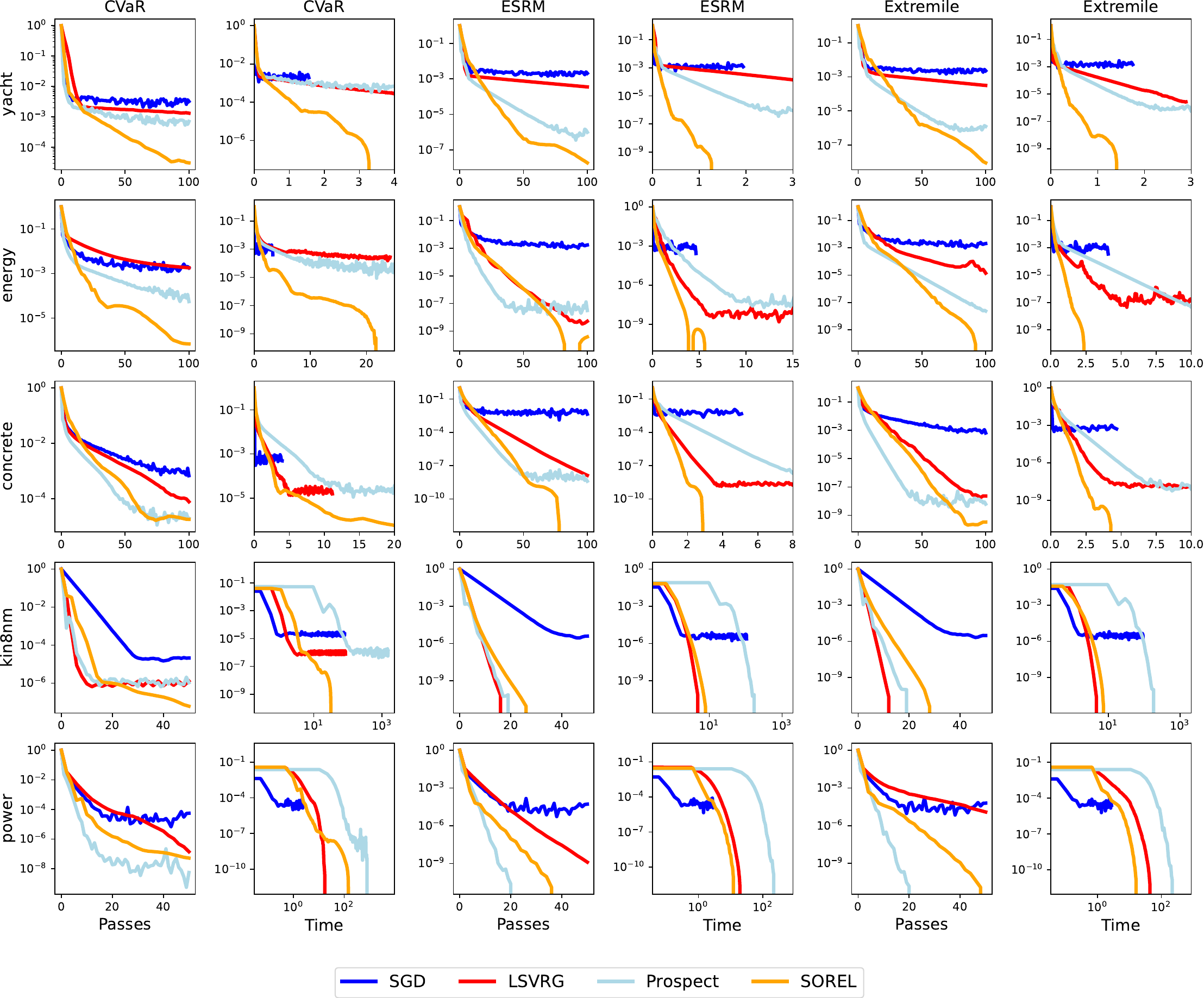}
    \caption{Suboptimality of spectral risks for different algorithms \textbf{without mini-batching}. The $x$-axis represents the effective number of samples used by the algorithm divided by $n$ (odd columns) or CPU time (even columns). Each row corresponds to the same dataset, and each column corresponds to the same type of the spectral risk.}
    \label{fig:regression-single}
\end{figure}

\paragraph{Results.}
\begin{figure}[t]
    \centering
    \includegraphics[width=\linewidth]{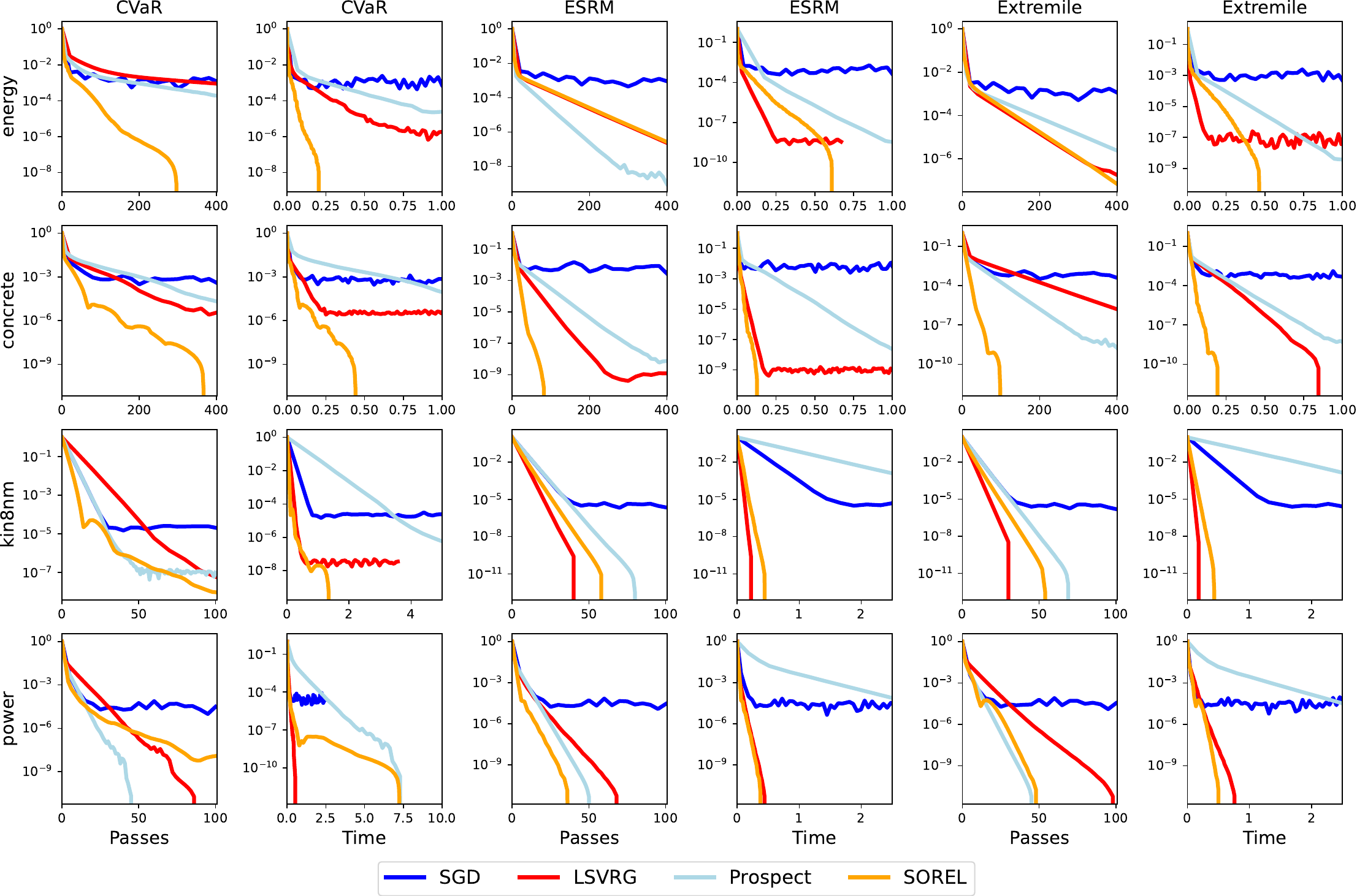}
    \caption{Suboptimality of spectral risks for different algorithms \textbf{with} \textbf{mini-batching}. The $x$-axis represents the effective number of samples used by the algorithm divided by $n$ (odd columns) or CPU time (even columns).}
    \label{fig:regression-batch}
\end{figure}

Figure~\ref{fig:regression-single} compares the training curves of our method with other baselines across various datasets and the spectral risk settings.
In terms of sample complexity and runtime, SOREL outperforms other baselines in most cases; SOREL also achieve comparable results in the \texttt{kin8nm} dataset.
In the \texttt{power} dataset, the sample complexity of Prospect is better than that of SOREL. However, the runtime of SOREL is significantly shorter than that of Prospect due to the fact that Prospect needs the calculation of projections onto the permutahedron with $O(n)$ operations each step. 
As expected, SGD fails to converge due to its inherent bias~\citep{mehta2022stochastic,levy2020largescale}. Although \citet{mehta2024distributionally} discusses the equivalence of minimizing the smoothed spectral risk and the original spectral risk when losses at the optimal point are different from each other, we find that LSVRG and Prospect often fail to reach the true optimal point, indicating limitations of these methods. In contrast, SOREL converges to the true optimal point in all settings (suboptimality less than 0 means the solution's accuracy is higher than L-BFGS).

Additionally, in Figure~\ref{fig:regression-batch}, we present the empirical results of the algorithms with minibatching. Minibatching has a significant improvement on the convergence rate of all the algorithms, even though we have not proven the convergence of them.
Similar to what is shown in Figure~\ref{fig:regression-single},
SGD, LSVEG and Prospect fail to converge to the true optimal points, especially in the first two datasets.
SOREL converges to the optimal solutions in all settings, and achieve the best or competitive results, in terms of sample complexity, runtime, or both, except in the setting of CVaR and \texttt{power} dataset. Still SOREL performs competitively for an accuracy of $10^{-7}$ in this setting.

\section{Conclusion}
We have proposed SOREL, the first stochastic algorithm with convergence guarantees for the spectral risk minimization problems. We have proved that SOREL achieves a near-optimal rate of $\widetilde{O}(1/\sqrt{\ep})$ in terms of $\ep$. In experiments, SOREL outperforms existing baselines in terms of sample complexity and runtime in most cases. The SOREL algorithm with minibatching also empirically achieves significant rate improvements.

Future work includes exploring convergence of SOREL for nonconvex problems, and investigating broader applications of the spectral risk in areas such as fairness and distributionally robust optimization.

\bibliographystyle{abbrvnat}
\bibliography{REFERENCES}

\newpage
\onecolumn
\begin{appendix}
\begin{center}
{\Large \bf Appendix}
\end{center}
\vspace{-0.1in}
\par\noindent\rule{\textwidth}{1pt}
\setcounter{section}{0}
\renewcommand\thesection{\Alph{section}}
\end{appendix}

\section{Proofs}\label{apx:proof}

First, we provide an auxiliary lemma. This is an extension of \citep[Lemma 8]{boob2024level}.

\begin{lem}\label{lem:auxi}
    Let $\Bar{\x}$ be an $\epsilon$-approximate solution of $\min_x \{g(\bx)+\frac{\la}{2}\|\bx-\hat{\bx}\|^2\}$ in expectation, where $g:\mathbb{R}^d\to\mathbb{R}$ is $\mu$-strongly convex, $\mu\geq 0$. Then
    $$
    \begin{aligned}
        \mathop{\mathbb{E}}\left\{g(\Bar{\bx})-g(\bx)\right\}\leq &\mathbb{E}\left\{\frac{\la}{2}\left[\|\bx-\hat{\bx}\|^2-\|\bx-\Bar{\bx}\|^2-\|\hat{\bx}-\Bar{\bx}\|^2\right]-\frac{\mu}{2}\|\bx-\Bar{\bx}\|^2\right\}\\
        &+\sqrt{\frac{(\la+\mu)\ep}{2}}\ex\|\Bar{\bx}-\bx\|^2+\sqrt{\frac{(\la+\mu)\ep}{2}}+\ep.
    \end{aligned}
    $$
\end{lem}
\begin{proof}
    Let $\bx^\star=\arg\min\limits_{x}\{g(\bx)+\frac{\la}{2}\|\bx-\hat{\bx}\|^2\}$. By $(\mu+\la)$-strong convexity of $g(\cdot)+\frac{\la}{2}\|\cdot-\hat{\bx}\|^2$
        \begin{align}
            &g(\bx)+\frac{\la}{2}\|\bx-\hat{\bx}\|^2\geq g(\bx^\star)+\frac{\la}{2}\|\bx^\star-\hat{\bx}\|^2+\frac{\mu+\la}{2}\|\bx-\bx^\star\|^2,\nonumber\\
            &g(\bx^\star)-g(\bx)\leq \frac{\la}{2}\left[\|\bx-\hat{\bx}\|^2-\|\bx^\star-\hat{\bx}\|^2-\|\bx^\star-\bx\|^2\right]-\frac{\mu}{2}\|\bx-\bx^\star\|^2\label{eq:sc}.
        \end{align}

    By the definition of $\Bar{\bx}$
    \begin{align}
        \ex\{ g(\Bar{\bx})+\frac{\la}{2}\|\Bar{\bx}-\hat{\bx}\|^2\}\leq g({\bx}^\star)+\frac{\la}{2}\|{\bx}^\star-\hat{\bx}\|^2 + \ep \label{eq:defi-approx}
    \end{align}
    Combining (\ref{eq:sc}) and (\ref{eq:defi-approx}) gives
    \begin{align}
        &\ex\left\{g(\Bar{\bx})-g({\bx})\right\}\leq\frac{\la}{2}\left[\|\bx-\hat{\bx}\|^2-\|\bx^\star-{\bx}\|^2-\ex\|\hat{\bx}-\bar{\bx}\|^2\right]-\frac{\mu}{2}\|\bx-\bx^\star\|^2+\ep\label{eq:combine-sc-defi}\\
        &=\ex \left\{\frac{\la}{2}\left[\|\bx-\hat{\bx}\|^2-\|\bx-\bar{\bx}\|^2-\|\Bar{\bx}-\hat{\bx}\|^2\right]-\frac{\mu}{2}\|\bx-\bar{\bx}\|^2\right.\nonumber\\
        &\quad\left.+\frac{\la+\mu}{2}\left[\|\bx-\Bar{\bx}\|^2-\|\bx-\bx^\star\|^2\right]+\ep\right\}\nonumber\\
        &\leq \ex \left\{\frac{\la}{2}\left[\|\bx-\hat{\bx}\|^2-\|\bx-\bar{\bx}\|^2-\|\Bar{\bx}-\hat{\bx}\|^2\right]-\frac{\mu}{2}\|\bx-\bar{\bx}\|^2\right.\nonumber\\
        &\quad\left.+(\la+\mu)\|\bx-\Bar{\bx}\|\|\Bar{\bx}-\bx^\star\|+\ep\right\}\label{eq:combine-sc-defi-result},\\ \nonumber
    \end{align}
    where the last inequality is due to the fact that $\|a\|^2-\|b\|^2\leq -2\langle a, b-a\rangle\leq 2\|a\|\|b-a\|$.

     Let $\bx=\Bar{\bx}$ in (\ref{eq:combine-sc-defi}), and take the expectation with respect to $\Bar{\bx}$. Then we have
     $$\frac{\la+\mu}{2}\ex\,\|{\bx}^\star-\Bar{\bx}\|^2\leq \ep.$$
    By Hölder's inequality we have
    $$
    \begin{aligned}
     &\ex \|\x-\Bar{\x}\|\|\Bar{\x}-\x^\star\|\leq \left(\ex\|\x-\Bar{\x}\|^2\right)^{\half}\left(\ex\|\x^\star-\Bar{\x}\|^2\right)^{\half}\\
     &\leq \half\left(\ex\|\x^\star-\Bar{\x}\|^2\right)^{\half}\left(1+\ex\|\x-\Bar{\x}\|^2\right).\\
     &\leq \half\sqrt{\frac{2\ep}{\lambda+\mu}}\left(1+\ex\|\x-\Bar{\x}\|^2\right).
    \end{aligned}
    $$

      Combining the above results and (\ref{eq:combine-sc-defi-result}) we get the desired result.
\end{proof}

Consider solving the problem from Line \ref{algline:inner_loop_start} to Line \ref{algline:inner_loop_end} in Algorithm~\ref{alg:main_alg} while updating $\bw$,
\begin{equation}\nonumber\label{eq:w_subproblem}
\begin{aligned}\min_{\bw}P_k(\bw):=g(\bw)+\blam\+^\top\bell(\bw)+\frac{1}{2\eta_k}\|\bw-\bw\0\|^2
    \end{aligned}
\end{equation}
The following lemma provides the accuracy between  $\bw_{k+1}$ and $\arg\min_{\bw} P_k(\bw)$.

\begin{lem}\label{lem:w_subproblem}
    Let $P_k(\w):=\blam_{k+1}^{\top}\bell(\w)+g(\w)+\frac{1}{2\tau_k}\|\w-\w\0\|^2$. Set $\alpha<\frac{L}{4}$, $m_k=\Theta(\frac{L\tau_k}{\mu\tau_k+1})$ and $T_k=O(m_k\log\frac{1}{\ep})$ in Algorithm~\ref{alg:main_alg}. The overall sample complexity of obtaining an $\ep$-approximate solution such that $\ex P_k(\bw^{k+1})-\min_{\bw}P_k(\bw)\leq \epsilon$ is $O\left(\left(n+\frac{L\tau_k}{\mu\tau_k+1}\right)\log\frac{1}{\ep}\right)$. Moreover, we can set $\alpha=\frac{1}{12L}$ and $m_k=\frac{96L}{\mu+\tau_k^{-1}}+2$ in practice.
\end{lem}
\begin{proof}
    First note that $\blam^{k+1\top}\bell(\bw)$ is $L$-smooth since
    $$
    \|\sum_{i=i}^n\lam_{k+1,i}\ell_i(\x)-\sum_{i=1}^n\lam_{k+1,i}\ell_i(\y)\|\leq \sum_{i=1}^n\lam_{k+1,i}\|\ell_i(\x)-\ell_i(\y)\|\leq L\|\x-\y\|,
   $$
   for $\forall \x,\y\in\mathbb{R}^d$. In the last inequality we use $\sum_i\lam_{k+1,i}\leq 1$ due to $\blam\+\in\Pi_{\bsig}$ and $L$-smoothness of $\ell_i$. Moreover, it is not hard to see that $P_k(\bw)$ is $\mu+\tau_k^{-1}$-strongly convex. By \citep[Theorem 1]{xiao2014proximal} we get the desired result. 
\end{proof}

Next, for convenience, we consider that $T_k$ (will be determined in Theorem~\ref{thm:main-thm}) is large enough so that $\w_k$ is a $\delta_k$-near optimal solution of $P_k(\w)$, that is, $\ex_k P_k(\w\+)-\min_{\w}P_k(\w)\leq\ep$. Here, $\ex_k$ represents the conditional expectation with respect to the random samples used to compute $\w\+$ given $\w\0,\dots,\w_0$. Then, we can provide a one-step analysis of the outer loop of our algorithm. We use $L(\bw,\blam)=\blam^\top\bell(\bw)+g(\bw)$ in the analysis, which is natural because our algorithm iteratively  updates $\bw$ and $\blam$.

\begin{lem}\label{lem:one-step}
    Suppose Assumption~\ref{assumption:basic} holds. Let $\{\bw\0\}$ and $\{\blam\0\}$ be the sequences generated by Algorithm~\ref{alg:main_alg}. Then for any $\bw\in\mathbb{R}^d$ and $\blam\in \Pi_{\bsig}$, the following inequality holds,
\begin{equation*}
    \begin{aligned}
        &\ex_k \, \left\{L(\bw\+,\blam)-L(\bw,\blam\+)\right\}\\
        \leq &\ex_k\left\{\langle\blam-\blam\+,\bell(\bw\+)\rangle+\frac{1}{2\eta_k}\left[\|\blam-\blam\0\|^2-\|\blam-\blam\+\|^2-\|\blam\+-\blam\0\|^2\right]\right.\\
        &+\frac{1}{2\tau\0}\left[\|\bw-\bw\0\|^2-\|\bw-\bw\+\|^2-\|\bw\+-\bw\0\|^2\right]-\frac{\mu}{2}\|\bw-\bw\+\|^2\\
        &+ \left.\langle \bv\0,\blam\+ -\blam \rangle +\delta\0+\sqrt{\frac{(\tau_k^{-1}+\mu)\delta_k}{2}}(\|\bw-\bw\+\|^2+1)\right\},
    \end{aligned}
\end{equation*}
where $\ex_k$ represents the conditional expectation with respect to the random samples used to compute $\bw\+$ given $\bw\0,\dots,\bw_0$.
\end{lem}
\begin{proof}
    From the update of $\blam_{k+1}$ and Lemma~\ref{lem:auxi} we have
    \begin{equation}\label{eq:y-update}
        0\leq\frac{1}{2\eta\0}\left[\|\blam-\blam\0\|^2-\|\blam-\blam\+\|^2-\|\blam\+-\blam\0\|^2\right] + \langle \bv\0,\blam\+-\blam\rangle.
    \end{equation}

From the update of $\bw\+$ and Lemma~\ref{lem:auxi} we have
\begin{align}\label{eq:x-update}
    &\ex_k \left\{g(\bw\+)+\langle\blam\+,\bell(\bw\+)\rangle-g(\bw)-\langle\blam\+,\bell(\bw)\rangle\right\}\nonumber\\
    &\leq \ex_k \left\{\frac{1}{2\tau\0}\left[\|\bw-\bw\0\|^2-\|\bw-\bw\+\|^2-\|\bw\+-\bw\0\|^2\right]-\frac{\mu}{2}\|\bw-\bw\+\|^2\nonumber\right.\\
    &\left.\quad+\de\0+\sqrt{\frac{(\tau_k^{-1}+\mu)\delta_k}{2}}(\|\bw-\bw\+\|^2+1)\right\}.
\end{align}

Taking the conditional expectation $\ex_k$ of both sides of (\ref{eq:y-update}) and summing with (\ref{eq:x-update}) we obtain
\begin{equation}
    \begin{aligned}
        &\ex_k \,\left\{ L(\bw\+,\blam)-L(\bw,\blam\+)\right\}\\
        =&\ex_k\, \left\{g(\bw\+)+\langle\blam,\bell(\bw\+)\rangle-g(\bw)-\langle\blam\+,\bell(\bw)\rangle\right\}\\
        \leq &\ex_k\,\left\{\langle\blam-\blam\+,\bell(\bw\+)\rangle+\frac{1}{2\eta_k}\left[\|\blam-\blam\0\|^2-\|\blam-\blam\+\|^2-\|\blam\+-\blam\0\|^2\right]\right.\\
        &  +\frac{1}{2\tau\0}\left[\|\bw-\bw\0\|^2-\|\bw-\bw\+\|^2-\|\bw\+-\bw\0\|^2\right]+\langle \bv\0,\blam\+ -\blam \rangle\\
        &\left. -\frac{\mu}{2}\|\bw-\bw\+\|^2+\delta\0+\sqrt{\frac{(\tau_k^{-1}+\mu)\delta_k}{2}}(\|\bw-\bw\+\|^2+1)\right\}.
    \end{aligned}
\end{equation}
\end{proof}

\begin{lem}\label{lem:recursive-descent}
Under the same assumptions as Lemma~\ref{lem:one-step}, for any $\bw\in\mathbb{R}^d$ and $\blam\in \Pi_{\bsig}$, we have
    \begin{equation}\label{eq:lem-descent}
    \begin{aligned}
        &\ex_k \left\{L(\bw\+,\blam)-L(\bw,\blam\+)\right\}\\
        \leq &\ex_k\left\{  \frac{1}{2\eta_k}\left[\|\blam-\blam_k\|^2-\|\blam-\blam\+\|^2\right]
         +\frac{1}{2\tau_k}\|\bw-\bw\0\|^2-\frac{1}{2}\left(\frac{1}{\tau_k}+\mu\right)\|\bw-\bw\+\|^2\right.  \\
         & + \langle \bell(\bw\+)-\bell(\bw\0),\blam-\blam\+\rangle
        - \theta_k \langle \bell(\bw\0)-          \bell(\bw\m),\blam-\blam\0\rangle\\
        & - \frac{1}{2}\left[\frac{1}{\eta_k}-\theta_k\frac{G}{\alpha_k}\right]\|\blam\0-\blam\+\|^2-\frac{1}{2\tau_k}\|\bw\0-\bw\+\|^2+\frac{G\theta_k\alpha_k}{2}\|\bw\0-\bw\m\|^2\\
        &  \left.
        + \delta_k+\sqrt{\frac{(\tau_k^{-1}+\mu)\delta_k}{2}}(\|\bw-\bw\+\|^2+1)\right\}.
    \end{aligned}
\end{equation}
\end{lem}
\begin{proof}
    First, we have \begin{equation}\nonumber
        \begin{aligned}
            &\langle \bv_k,\blam\+-\blam\rangle\\ =&\langle\bell(\bw\0)+\theta_k\left(\bell(\bw\0)-\bell(\bw\m)\right),\blam\+-\blam\rangle\\
            =&-\langle\bell(\bw\+),\blam-\blam\+\rangle+\langle \bell(\bw\+)-\bell(\bw\0),\blam-\blam\+\rangle\\
            \quad& -\theta_k \langle \bell(\bw\0)-\bell(\bw\m),\blam-\blam\0\rangle-\theta_k \langle \bell(\bw\0)-\bell(\bw\m),\blam\0-\blam\+\rangle.\\
        \end{aligned}
    \end{equation}
    Then we obtain that
    \begin{equation}\label{eq:pri-dual-ineq}
        \begin{aligned}
        &\langle\blam-\blam\+,\bell(\bw\+)\rangle+\langle \bv_k,\blam\+-\blam\rangle\\
        &\leq
        \langle \bell(\bw\+)-\bell(\bw\0),\blam-\blam\+\rangle-\theta_k \langle \bell(\bw\0)-\bell(\bw\m),\blam-\blam\0\rangle\\
            &\quad-\theta_k \langle \bell(\bw\0)-\bell(\bw\m),\blam\0-\blam\+\rangle.
    \end{aligned}
    \end{equation}

    Next we bound the last term on the right-hand side of (\ref{eq:pri-dual-ineq}).
    \begin{equation}\label{eq:smooth-phi}
        \begin{aligned}
            &\langle \bell(\bw\0)-\bell(\bw\m),\blam\0-\blam\+\rangle\\
            \leq& G\|\bw\0-\bw\m\|\|\blam\0-\blam\+\|\\
            \leq& \frac{G\alpha_k}{2}\|\bw\0-\bw\m\|^2+\frac{G}{2\alpha_k}\|\blam\0-\blam\+\|^2,
        \end{aligned}
    \end{equation}
    where the first inequality is due to the Lipschitz conditions of $\ell_i$ and in the second inequality we use Young's inequality with $\alpha_k> 0$.

    Combing (\ref{eq:pri-dual-ineq}) and (\ref{eq:smooth-phi}) we get
    \begin{equation}\label{eq:smooth-phi-result}
        \begin{aligned}
            &\langle\blam-\blam\+,\bell(\bw\+)\rangle+\langle \bv_k,\blam\+-\blam\rangle\\
        &\leq \langle \bell(\bw\+)-\bell(\bw\0),\blam-\blam\+\rangle-\theta_k \langle \bell(\bw\0)-\bell(\bw\m),\blam-\blam\0\rangle\\
            &\quad +\frac{G\alpha_k\theta_k}{2}\|\bw\0-\bw\m\|^2+\frac{G\theta_k}{2\alpha_k}\|\blam\0-\blam\+\|^2.
        \end{aligned}
    \end{equation}
    Taking the conditional expectation $\ex_k$ of both sides of (\ref{eq:smooth-phi-result}) and combing it with Lemma~\ref{lem:one-step}
    we get the desired result.
\end{proof}

We remark that $\alpha_k$ does not need to be computed in the actual algorithm but only exists in the theoretical analysis. Next, we try to telescope the terms on the right hand side of (\ref{eq:lem-descent}) by multiplying each term by $\gamma_k$. To ensure that the adjacent terms in the sequence $k=0,\dots,K-1$ can be canceled out during summation, we need the parameters of the algorithm to satisfy the following conditions.
\begin{con}\label{condition:alg_para}
For $k=0,1,...$, the following conditions for parameters in the analysis and Algorithm~\ref{alg:main_alg}:
    \begin{subequations}
        \begin{align}
            &\frac{\gamma_{k+1}}{\eta_{k+1}}\leq \frac{\gamma_k}{\eta_k},\label{assumption:step-size_1}\\
            &\frac{\gamma_{k+1}}{\tau_{k+1}}\leq \gamma_k\left(\frac{1}{\tau_k}+\mu-\sqrt{2(\mu+\tau_k^{-1})\delta_k}\right),\label{assumption:step-size_2}\\
            & \gamma_k=\gamma_{k+1}\theta_{k+1},\label{assumption:step-size_3}\\
            & G\alpha_{k+1}\leq \frac{1}{\tau_k},\label{assumption:step-size_6}\\
            & \theta_k\frac{G}{\alpha_k}\leq\frac{1}{\eta_k}.\label{assumption:step-size_7}
        \end{align}
    \end{subequations}
\end{con}

\begin{lem}\label{lem:pri-convergence}
    Assume Assumption~\ref{assumption:basic} holds and Condition~\ref{condition:alg_para} is satisfied. Then for all $\bw\in\mathbb{R}^d$ and $\blam\in \Pi_{\bsig}$ we have
    $$
    \frac{\gamma_K}{2\tau_K}\ex\|\bw^\star-\bw_K\|^2\leq  \frac{\gamma_0}{2\eta_0}\|\blam^\star-\blam_0\|^2 + \frac{\gamma_0}{2\tau_0}\|\bw^\star-\bw_0\|^2
        +  \sum_{k=0}^{K-1}\left(\delta_k \gamma_k+\frac{\gamma_k}{2}\sqrt{2(\mu+\tau_k^{-1})\delta_k}\right),$$
    where $\bw^\star=\arg\min_{\bw}R_{\bsig}(\bw)+g(\bw)$ and $\blam^\star=\bsig_{\pi^{-1}}$. $\pi$ is the permutation that arranges $\ell_1(\bw^\star),\dots,\ell_n(\bw^\star)$ in ascending order, that is, $\ell_{\pi(1)}(\bw^\star)\leq\dots\leq\ell_{\pi(n)}(\bw^\star)$.
\end{lem}
\begin{proof}
    Taking  expectations with respect to $\bw\0,\dots,\bw_1$
    in (\ref{eq:lem-descent}) and using the law
    of total expectation yields
\begin{equation}\label{eq:one-step-final}
    \begin{aligned}
        &\ex \left\{L(\bw\+,\blam)-L(\bw,\blam\+)\right\}\\
        \leq \, &\ex\, \left\{\frac{1}{2\eta_k}\left[\|\blam-\blam_k\|^2-\|\blam-\blam\+\|^2\right] -\frac{1}{2}\left[\frac{1}{\eta_k}-\frac{G\theta_k}{\alpha_k}\right]\|\blam\0-\blam\+\|^2\right.\\
        & +\langle \bell(\bw\+)-\bell(\bw\0),\blam-\blam\+\rangle -\theta_k \langle \bell(\bw\0)-          \bell(\bw\m),\blam-\blam\0\rangle\\
        & +\frac{1}{2\tau_k}\|\bw-\bw\0\|^2-\frac{1}{2}\left(\frac{1}{\tau_k}+\mu-\sqrt{2(\mu+\tau_k^{-1})\delta_k}\right)\|\bw-\bw\+\|^2\\
        & \left.-\frac{1}{2\tau_k}\|\bw\0-\bw\+\|^2+\frac{G\theta_k\alpha_k}{2}\|\bw\0-\bw\m\|^2 +\delta_k+\frac{1}{2}\sqrt{2(\mu+\tau_k^{-1})\delta_k}\right\}\,.
    \end{aligned}
\end{equation}

Multiplying both sides of (\ref{eq:one-step-final}) by $\gamma_k$ and  summing over $k=0$ to $K-1$ we obtain that

\begin{equation}\nonumber
    \begin{aligned}
        &\sum_{k=0}^{K-1} \gamma_k\ex\left\{L(\bw\+,\blam)-L(\bw,\blam\+)\right\}\\
        \leq \,&\ex\left\{\ \frac{\gamma_0}{2\eta_0}\|\blam-\blam_0\|^2+\sum_{k=0}^{K-2}\underbrace{\half\left(\frac{\gamma\+}{\eta\+}-\frac{\gamma\0}{\eta\0}\right)\|\blam-\blam\+\|^2}_{A}-\frac{\gamma_{K-1}}{2\eta_{K-1}}\|\blam-\blam_K\|^2\right.\\
        +&\frac{\gamma_0}{2\tau_0}\|\bw-\bw_0\|^2+\sum_{k=0}^{K-2}\underbrace{\half\left[\frac{\gamma\+}{\tau\+}-\gamma_k\left(\frac{1}{\tau_k}+\mu-\sqrt{2(\mu+\tau_k^{-1})\delta_k}\right)\right]\|\bw-\bw\+\|^2}_{B}\\
        -& \frac{\gamma_{K-1}}{2}\left(\frac{1}{\tau_{K-1}}+\mu-\sqrt{2(\mu+\tau_{K-1}^{-1})\delta_{K-1}}\right)\|\bw-\bw_K\|^2\\
        +&\sum_{k=0}^{K-2} \underbrace{(\gamma_k-\gamma\+\theta\+)}_{C}\langle \bell(\bw\+)-\bell(\bw\0),\blam-\blam\+\rangle+\gamma_{K-1}\langle \bell(\bw_K)-\bell(\bw_{K-1}),\blam-\blam_K\rangle\\
        +&\half\sum_{k=0}^{K-2}\underbrace{\left(\gamma\+\theta\+\alpha\+ G-\frac{\gamma_k}{\tau_k}\right)}_{D}\|\bw\0-\bw\+\|^2-\frac{\gamma_{K-1}}{2\tau_{K-1}}\|\bw_K-\bw_{K-1}\|^2\\
        +&\left.\half\sum_{k=0}^{K-1}\underbrace{\left[-\gamma_k\left(\frac{1}{\eta_k}-\theta_k\frac{G}{\alpha_k}\right) \right] }_{E}\|\blam\0-\blam\+\|^2+\sum_{k=0}^{K-1}\left(\delta_k \gamma_k + \frac{\gamma_k}{2}\sqrt{2(\mu+\tau_k^{-1})\delta_k}\right)\right\}.\\
    \end{aligned}
\end{equation}
Here we use $ \bell(\bw_0)-\bell(\bw_{-1})=0$ by $\bw_0=\bw_{-1}$ and $\blam_0=\blam_{-1}$. By Condition~\ref{condition:alg_para}, we have $A,B,D,E \leq 0$ and $C=0$.

Then we have
\begin{equation}\label{eq:main-descent}
    \begin{aligned}
        &\sum_{k=0}^{K-1} \gamma_k\ex\left\{L(\bw\+,\blam)-L(\bw,\blam\+)\right\}\\
        \leq \, & \ex\left\{\frac{\gamma_0}{2\eta_0}\|\blam-\blam_0\|^2-\frac{\gamma_{K-1}}{2\eta_{K-1}}\|\blam-\blam_K\|^2+\frac{\gamma_0}{2\tau_0}\|\bw-\bw_0\|^2\right.\\
        -& \frac{\gamma_{K-1}}{2} \left(\frac{1}{\tau_{K-1}}+\mu-\sqrt{2(\mu+\tau_{K-1}^{-1})\delta_{K-1}}\right)\|\bw-\bw_K\|^2\\
        + & \gamma_{K-1}\langle \bell(\bw_K)-\bell(\bw_{K-1}),\blam-\blam_K\rangle-\frac{\gamma_{K-1}}{2\tau_{K-1}}\|\bw_K-\bw_{K-1}\|^2\\
        +& \left.\sum_{k=0}^{K-1}\left(\delta_k \gamma_k + \frac{\gamma_k}{2}\sqrt{2(\mu+\tau_k^{-1})\delta_k}\right)\right\}.
    \end{aligned}
\end{equation}

Next  we bound $\gamma_{K-1}\langle \bell(\bw_K)-\bell(\bw_{K-1}),\blam-\blam_K\rangle$ similar to (\ref{eq:smooth-phi}). We have
\begin{equation}\nonumber
        \begin{aligned}
            &\langle \bell(\bw_K)-\bell(\bw_{K-1}),\blam-\blam_K\rangle
            \leq \frac{G\alpha_K}{2}\|\bw_K-\bw_{K-1}\|^2+\frac{1}{2}\frac{G}{\alpha_K}\|\blam-\blam_K\|^2.
        \end{aligned}
    \end{equation}

Taking the expectation and plugging this into (\ref{eq:main-descent}) we obtain that
\begin{equation}
    \begin{aligned}
        &\sum_{k=0}^{K-1}\gamma_k\ex\, \left\{L(\bw\+,\blam)-L(\bw,\blam\+)\right\}\\
        \leq &\ex\, \left\{\frac{\gamma_0}{2\eta_0}\|\blam-\blam_0\|^2-\half\underbrace{\left[\frac{\gamma_{K-1}}{\eta_{K-1}}-\gamma_{K-1}\frac{G}{\alpha_K}\right]}_{\tilde{A}}\|\blam-\blam_K\|^2\right.\\
        +& \frac{\gamma_0}{2\tau_0}\|\bw-\bw_0\|^2-\underbrace{\frac{\gamma_{K-1}}{2}\left(\frac{1}{\tau_{K-1}}+\mu-\sqrt{2(\mu+\tau_{K-1}^{-1})\delta_{K-1}}\right)}_{\tilde{B}}\|\bw-\bw_K\|^2\\
        +&\left.\frac{\gamma_{K-1}}{2}\underbrace{\left(\alpha_K G-\frac{1}{\tau_{K-1}}\right)}_{\tilde{C}}\|\bw_K-\bw_{K-1}\|^2+\sum_{k=0}^{K-1}\left(\delta_k \gamma_k+\frac{\gamma_k}{2}\sqrt{2(\mu+\tau_k^{-1})\delta_k}\right)\right\}\\
    \end{aligned}
\end{equation}

We analyze $\Tilde{A}$-$\Tilde{D}$ under Condition~\ref{condition:alg_para}:
$$\begin{aligned}
    &\tilde{A}\overset{(\ref{assumption:step-size_1})}{\geq} \left[\frac{\gamma_K}{\eta_K}-\gamma_{K-1}\frac{G}{\alpha_K}\right]\overset{(\ref{assumption:step-size_3})}{=}\gamma_K\left[\frac{1}{\eta_K}-\theta_{K}\frac{G}{\alpha_K}\right]\overset{(\ref{assumption:step-size_7})}{\geq} 0,\\
    &\tilde{B}\overset{(\ref{assumption:step-size_2})}\geq \frac{\gamma_K}{2\tau_K},\\
    &\tilde{C} \overset{(\ref{assumption:step-size_6})}{\leq} 0.
\end{aligned}$$

We obtain that
\begin{equation}\nonumber
    \begin{aligned}
        &\sum_{k=0}^{K-1}\gamma_k\ex\, \left\{L(\bw\+,\blam)-L(\bw,\blam\+)\right\}\\
        \leq & \frac{\gamma_0}{2\eta_0}\|\blam-\blam_0\|^2 + \frac{\gamma_0}{2\tau_0}\|\bw-\bw_0\|^2 - \frac{\gamma_K}{2\tau_K}\ex\,\|\bw-\bw_K\|^2+\sum_{k=0}^{K-1}\left(\delta_k \gamma_k+\frac{\gamma_k}{2}\sqrt{2(\mu+\tau_k^{-1})\delta_k}\right)\\
    \end{aligned}
\end{equation}

 Let $\bw=\bw^\star,\blam=\blam^\star$, for any $\bw\in\mathbb{R}^d$ and $\blam\in\Pi_{\bsig}$, we have $L(\bw^\star,\blam^\star)=\max_{\blam\in\Pi_{\bsig}}L(\bw^\star,\blam)\geq L(\bw^\star,\blam)$. On the other hand, we have $L(\bw,\blam^\star)\geq L(\bw^\star,\blam^\star)=\min_{\bw}L(\bw,\blam^\star)$. Thus we obtain that $L(\bw\+,\blam)-L(\bw,\blam\+)\geq 0$ for $\forall k=0,\dots,K-1$. This completes the proof.
\end{proof}

Now we are ready to present our main theorem.
By choosing appropriate parameters in Algorithm~\ref{alg:main_alg} to satisfy Condition~\ref{condition:alg_para}, we can achieve the desired convergence rate.

\begin{theo}\label{thm:main-thm}
    Assume Assumption~\ref{assumption:basic} holds.
    Set $\eta_k=\frac{\mu(k+1)}{8G^2},\theta_k=\frac{k}{k+1},\gamma_k=k+1,\tau_k=\frac{4}{\mu (k+1)},\alpha_{k+1} = G\eta_k,\delta_k = \min \left(\frac{\mu}{8(k+5)},\mu(k+1)^{-6}\right)$. 
    Then Condition~\ref{condition:alg_para} is satisfied by the above parameters. Moreover, use the same conditions as Lemma~\ref{lem:w_subproblem}, while replacing $\epsilon$ with $\delta_k$ for $k=0,1,\dots,K-1$. Then we have
     $$
     \ex \|\bw_K-\bw^\star\|^2=O\left(\frac{G^2}{\mu^2K^2}\right).$$
\end{theo}

\begin{proof}
    First, we obtain an $\delta_k$ approximate solution to (\ref{eq:w_subproblem}) through $T_k$ updates to $\bw$ in Algorithm~\ref{alg:main_alg} by Lemma~\ref{lem:w_subproblem}.
    We then verify that Condition~\ref{condition:alg_para} is satisfied by the parameters.

    It is not hard to see that $\frac{\gamma\+}{\gamma\0}=\frac{\eta_{k+1}}{\eta_k}=\frac{k+2}{k+1}$ and $\theta_{k+1}=\frac{\gamma_k}{\gamma_{k+1}}=\frac{k+1}{k+2}$. Thus (\ref{assumption:step-size_1}) and (\ref{assumption:step-size_3}) are satisfied.

    Since $\delta_k\leq\frac{\mu}{8(k+5)}$, we have $\sqrt{2(\mu+\tau_k^{-1})\delta_k}=\sqrt{2\mu(1+\frac{k+1}{4})\delta_k}\leq \frac{\mu}{4}$. Then we obtain that
    $$\frac{\gamma_{k+1}}{\gamma_k\tau_{k+1}}=\frac{k+2}{4}\mu+\frac{k+2}{4(k+1)}\mu\leq \frac{k+4}{4}\mu,$$
    and
    $$
    \frac{1}{\tau_k}+\mu-\sqrt{2(\mu+\tau_k^{-1})\delta_k}\geq\frac{k+1}{4}\mu+\mu-\frac{\mu}{4}=\frac{k+4}{4}\mu.$$
    Thus (\ref{assumption:step-size_2}) holds.

    Furthermore, (\ref{assumption:step-size_6}) and (\ref{assumption:step-size_7}) hold due to $G\alpha_{k+1}=G^2\eta_k=\frac{k+1}{8}\mu\leq\frac{k+1}{4}\mu=\frac{1}{\tau_k}$ and $\theta_k\frac{G}{\alpha_k}=\frac{\eta_{k-1}}{\eta_k}\frac{G}{G\eta_{k-1}}=\frac{1}{\eta_k}$.

    Now Condition~\ref{condition:alg_para} is satisfied. By Lemma~\ref{lem:pri-convergence}, we have
    $$
    \frac{\gamma_K}{2\tau_K}\ex\,\|\bw-\bw_K\|^2\leq \frac{\gamma_0}{2\eta_0}\|\blam^\star-\blam_0\|^2+\frac{\gamma_0}{2\tau_0}\|\bw^\star-\bw_0\|^2 + \sum_{k=0}^{K-1}\left(\delta_k \gamma_k+\frac{\gamma_k}{2}\sqrt{2(\mu+\tau_k^{-1})\delta_k}\right).$$

    Since $\delta_k\leq \mu(k+1)^{-6}$, we have $\sum_{k=0}^\infty \delta_k \gamma_k\leq \mu\sum_{k=0}^\infty (k+1)^{-5}\leq \frac{\mu}{4}$, 
    and 
    $$\sum_{k=0}^\infty \gamma_k\sqrt{(\mu+\tau_k^{-1})\delta_k}\leq\frac{\sqrt{2}\mu}{4}\sum_{k=0}^\infty (k+1)^{-2}\sqrt{k+5}\leq \frac{\sqrt{2}\mu}{4}\sum_{k=0}^\infty (k+1)^{-2}\left(\sqrt{k+1}+2\right)\leq \sqrt{2}\mu.$$

    Finally, by $\frac{\gamma_K}{2\tau_K}=\frac{\mu (K+1)^2}{8},\tau_0=\frac{4}{\mu}$ and $\eta_0=\frac{\mu}{8G^2}$, we get the desired result.
\end{proof}

\begin{coro}
    Suppose the step-size $\alpha<\frac{L}{4}$ and $m_k=\Theta(\frac{L}{\mu(k+1)})$. With the same conditions as Theorem~\ref{thm:main-thm}, we obtain an output $\bw_K$ such that $\ex\|\bw_K-\bw^\star\|^2\leq\epsilon$ in a total sample complexity of $O\left(n\frac{G}{\mu\sqrt{\epsilon}}\log\frac{G}{\mu^2\sqrt{\epsilon}}+\frac{L}{\mu}\log\frac{G}{\mu\sqrt{\epsilon}}\log\frac{G}{\mu^2\sqrt{\epsilon}}\right)$ using Algorithm~\ref{alg:main_alg}.
\end{coro}
\begin{proof}
    Recall that $\tau_k=\frac{4}{\mu(k+1)}$. It is not hard to see that $\frac{L\tau_k}{\mu\tau_k+1}=\frac{4L}{\mu(k+5)}\leq \frac{4L}{\mu(k+1)}$. By Lemma~\ref{lem:w_subproblem}, we get a $\delta_k$ approximate solution with the sample complexity of $C_{\bw\+}=O\left(\left(n+\frac{L}{\mu(k+1)}\right)\log\left(\delta_k^{-1}\right)\right)$. We set $\delta_k=\min\left(\frac{\mu}{8(k+5)},\mu(k+1)^{-6}\right)=\mu(k+1)^{-6}$ for $k\geq 1$. And $\delta_0=\mu/40$. The total sample complexity is
    $$
    \begin{aligned}
        &\sum_{k=0}^{K-1}C_{\bw\+}=\sum_{k=0}^{K-1}O\left(\left(n+\frac{L}{\mu(k+1)}\right)\left(\log(k+1)-\log\mu\right)\right)\\
        &=O\left(nK\log\frac{K}{\mu}+\frac{L}{\mu}\log K\log\frac{K}{\mu} \right).
    \end{aligned}
    $$
    In the last equality, we calculate $\sum_{k=1}^{K}\frac{\log k}{k}=O\left(\left(\log K\right)^2\right)$, $\sum_{k=1}^K \log k=O\left(K\log K\right)$ and $\sum_{k=1}^K\frac{1}{k}=O(\log K)$. By Theorem~\ref{thm:main-thm}, to achieve an $\epsilon$ approximate solution, we need $K=O\left(\frac{G}{\mu\sqrt{\ep}}\right)$. Therefore, the total sample complexity is $O\left(n\frac{G}{\mu\sqrt{\epsilon}}\log\frac{G}{\mu^2\sqrt{\epsilon}}+\frac{L}{\mu}\log\frac{G}{\mu\sqrt{\epsilon}}\log\frac{G}{\mu^2\sqrt{\epsilon}}\right).$
\end{proof}

\section{Experimental Details}\label{apx:exp-detail}
We now outline the details of our experimental setup. Our experimental setup mainly follows that of~\citet{mehta2024distributionally}.

\paragraph{Datasets.}
We use the same five datasets from the regression task in~\citet{mehta2024distributionally}. The statistical characteristics are summarized in Table~\ref{tab:dataset_summary}. Tasks for each dataset are as follows:
\begin{enumerate}
    \item \texttt{yacht}: predicting the residual resistance of sailing yachts based on their physical attributes.
    \item \texttt{energy}: predicting the cooling load of buildings based on their physical attributes.
    \item \texttt{concrete}: predicting the compressive strength of concrete types based on their physical and chemical properties.
    \item \texttt{kin8nm}: predicting the distance of a robot arm, consisting of 8 fully rotating joints, to a target point in space.
    \item \texttt{power}: predicting the net hourly electrical energy output of a power plant based on environmental factors.
\end{enumerate}

\begin{table}[h!]
    \centering
    \begin{tabular}{lccc}
        \toprule
        Dataset & \# features & \# samples & Source \\
        \midrule
        \texttt{yacht} & 6 & 244 & \citet{tsanas2012accurate} \\
        \texttt{energy} & 8 & 614 &  \citet{baressi2019artificial} \\
        \texttt{concrete} & 8 & 824 &  \citet{yeh2006analysis} \\
        \texttt{kin8nm} & 8 & 6,553  & \citet{akujuobi2017delve} \\
        \texttt{power} & 4 & 7,654 & \citet{tüfekci2014prediction} \\
        \bottomrule
    \end{tabular}
    \caption{ Statistical details of five real datasets and sources.}
    \label{tab:dataset_summary}
\end{table}
In the experiments, the sample matrix $\boldsymbol{X}\in\mathbb{R}^{n\times d}$ of the training set is standardized such that each column has a zero mean and a unit variance.

\paragraph{Objectives.}
We use linear models in our experiments. For spectral risks, we adopt three types: ESRM ($\rho=2$), Extremile ($r=2.5$), and CVaR ($\alpha=0.5$), as specified in Table~\ref{tab:sigma}. Additionally, we set the regularizer $g(\bw)$ to $\frac{\mu}{2}\|\bw\|^2$ with $\mu=\frac{1}{n}$. Thus, Problem~\ref{eq:problem} can be written as
$$\min_{\bw} \sum_{i=1}^n \sigma_i\ell_{[i]}(\bw)+\frac{\mu}{2}\|\bw\|^2,$$
where $\ell_i(\bw)=\frac{1}{2}\|y_i-\bw^\top\boldsymbol{x}_i\|^2$ and $(\boldsymbol{x}_i,y_i)$ are samples from the training set.

\paragraph{Hyperparameter Selection.}
\ We use the same hyperparameter selection method as in~\citep{mehta2024distributionally}. We set the batch size for SGD to $64$. Each time, we randomly select a minibatch indexes from $\{1,\dots,n\}$ without replacement for SGD and randomly select a single index for other methods. For the selection of step size $\alpha$, we set the random seed $s\in\{1,\dots,S\}$. For a single seed $s$, we calculate the average training loss of the last ten epochs, donated by $L_s(\alpha)$. We choose $\alpha$ that minimizes $\frac{1}{S}\sum_{s=1}^SL_s(\alpha)$, where $\alpha\in\{1 \times 10^{-4}, 3 \times 10^{-4}, 1 \times 10^{-3}, 3 \times 10^{-3}, 1 \times 10^{-2}, 3 \times 10^{-2}, 1 \times 10^{-1}, 3\times 10^{-1}\}$. For LSVRG, we set the length of an epoch to $n$. For SOREL,
we set $T_k=m_k=n$. Moreover, we set batch size to 64 for all algorithms with minibatching.

For SOREL, we follow the parameter values given in Theorem~\ref{thm:main-thm-paper}. In particular, we set $\theta_k=\frac{k}{k+1}$ and $\tau_k=\frac{20n}{k+1}$ in all experiments. Therefore, there are only two parameters $\alpha$ and $\eta_k$ left to tune. We set $\eta_k=\frac{C(k+1)}{n}$ and choose $C$ from $\{ 1 \times 10^{-2}, 2 \times 10^{-2}, 4 \times 10^{-2}, 1 \times 10^{-1}, 2 \times 10^{-1}, 4 \times 10^{-1}, 1 \times 10^{0}, 2 \times 10^{0}, 4 \times 10^{0}\}$, since the Lipschitz constant $G$ is hard to estimate. We use grid search to select $\alpha$ and $C$, with the selection criteria being the same as the previous paragraph.

\paragraph{Experimental Environment.}
We run all experiments on a laptop with 16.0 GB RAM and Intel i7-1360P 2.20 GHz CPU. All algorithms are implemented in Python 3.8.

\section{Examples}~\label{apx:example}
To illustrate the necessity of stabilizing the trajectory of the primal variable in Section~\ref{sec:alg-prob}, we provide a toy example. For simplicity, we consider a one-dimensional problem
\begin{equation}\label{eq:example}
    \min_{w\in\mathbb{R}} \sigma_1\ell_{[1]}(w)+\sigma_2\ell_{[2]}(w),
\end{equation}
where $\sigma_1 = 0, \sigma_2=1$ and $\ell_1 = \frac{1}{2}(w-1)^2,\ell_2(w)=\frac{1}{2}(w+1)^2$.
We use the following deterministic method, similar to Algorithm~\ref{alg:main_alg}.

\begin{example}
    For any $0<\alpha<2$, suppose we solve Problem~(\ref{eq:example}) using Algorithm~\ref{alg:example}
    and $T$ is sufficiently large. In that case, the iterative sequence $\{w_k\} $ can not converge to the optimal solution for any initial point $w_0$.
\end{example}

\begin{algorithm}
\caption{Simplified Algorithm for Solving the Example Problem. }
\label{alg:example}
\begin{algorithmic}[1]
\FORALL{$k = 0, 1, \dots$}
    \STATE Update $\{\lambda_{k+1,1}, \lambda_{k+1,2}\} = \{\sigma_1,\sigma_2\}$ if $\ell_1(w_k)\geq\ell_2(w_k)$ else $\{\sigma_2,\sigma_1\}$. Set $w_k^0 = w_k$.
    \FORALL{$t = 0, 1, \ldots, T-1$}
        \STATE  Compute the gradient $g^t = \lambda_{k+1,1} \nabla \ell_1(w_k^t) + \lambda_{k+1,2} \nabla \ell_2(w_k^t)$.
        \STATE  Update $w^{t+1} = w^t - \alpha g^t$.
    \ENDFOR
    \STATE Set $w_{k+1} = w_k^T$.
\ENDFOR
\end{algorithmic}
\end{algorithm}

Without loss of generality, we assume $w_0 > 0$, in which case $\blam_1 = [0,1]^\top$. We solve $\min_{w_1\in\mathbb{R}} \frac{1}{2}(w_1+1)^2$ through sufficient steps of gradient descent to obtain $w_1 = -1$. At this point, $\blam_2 = [1,0]^\top$. By iterating this process, $w_k$ always oscillates between -1 and 1, unable to converge to $w^\star = 0$. If $w_0 = 0$, we set $\sigma_1 = 1$ and $\sigma_2 = 0$, reaching the same conclusion. A similar conclusion can be extended to stochastic methods in the expectation sense.

We know that $\ell_1(w^\star)=\ell_2(w^\star)$ at the optimal point $w^\star=0$. Clearly, the iterative sequence of the algorithm oscillates at $w^\star$ and cannot converge to the optimal solution.
Although \citet{mehta2022stochastic,mehta2024distributionally} employ a similar approach to update $\blam$ for subgradient estimations, they consider the smoothed spectral risk by adding a strongly concave term with respect to $\blam$. However, for the original spectral risk minimization problems, updating $\blam$ with their method results in discontinuities, thereby lacking convergence guarantees.

\end{document}